\pdfoutput=1

\documentclass[english]{jnsao}
\usepackage[utf8]{inputenc}

\providecommand{\Div}{\operatorname{div}}          

\providecommand{\Det}{\operatorname{det}}                    

\providecommand{\argmin}{\operatorname*{argmin}}  
\providecommand{\Id}{\Op{Id}}                     

\providecommand{\limsup}{\operatorname*{limsup}}

\newcommand{\VR}{{\mathbf{R}}}

\newcommand{\Dsf}{\mathsf{D}}

\newcommand*{\Op}[1]{\mathsf{#1}} 

\usepackage{enumitem}

\setenumerate[0]{label=(\alph*)}

\newcommand{\overbar}[1]{\mkern 1.5mu\overline{\mkern-1.5mu#1\mkern-1.5mu}\mkern 1.5mu}

\newcommand{\beqn}{\begin{eqnarray*}}
\newcommand{\eeqn}{\end{eqnarray*}}

\newcommand{\ben}{\begin{equation}}
\newcommand{\een}{\end{equation}}

\newcommand{\beq}{\begin{eqnarray}}
\newcommand{\eeq}{\end{eqnarray}}

\newcommand{\benn}{\begin{equation*}}
\newcommand{\eenn}{\end{equation*}}

\usepackage[nameinlink,capitalize]{cleveref}

\newtheorem{assumption}[theorem]{Assumption}
\newtheorem*{assumption*}{Assumption}

\title{First-order differentiability properties of a class of equality constrained optimal value functions with applications}

\author{Kevin Sturm%
\thanks{Institut für Analysis und Scientific Computing, TU Wien, Wiedner Hauptstraße 8--10,
1040 Wien, Austria (\email{kevin.sturm@asc.tuwien.ac.at})}}
\shortauthor{Sturm}

\manuscriptcopyright{{\copyright} the author}
\manuscriptlicense{CC-BY-SA 4.0}

\manuscriptsubmitted{2020-01-16}
\manuscriptaccepted{2020-11-25}
\manuscriptvolume{1}
\manuscriptnumber{6034}
\manuscriptyear{2020}
\manuscriptdoi{10.46298/jnsao-2020-6034}

\begin{document}

\maketitle

\begin{abstract}
    In this paper we study the right differentiability of a parametric infimum function over a parametric set defined by equality constraints. We present a new theorem with sufficient conditions for the right differentiability with respect to the parameter. Target applications are nonconvex objective functions with equality constraints arising in optimal control and shape optimisation. The theorem makes use of the averaged adjoint approach in conjunction with the variational approach of Kunisch, Ito and Peichl. We provide two examples of our abstract result: (a) a shape optimisation problem involving a semilinear partial differential equation which exhibits infinitely many solutions, (b) a finite dimensional quadratic function subject to a nonlinear equation.
\end{abstract}


\section{Introduction}
Let a normed space $X$, a vector space $Y$ and $\tau >0$ be given. In this paper we study the one-sided differentiability in $t=0^+$ of the optimal value-function
\ben\label{E:gt}
t\mapsto g(t):= \inf_{u  \in E(t)} f(t,u ),
\een
where $f:[0,\tau]\times X\to \VR$ is a given function. The $E(t)$ denotes the set of states given by
\ben
E(t) = \{u^t\in X: e(t,u^t,\varphi)=0 \quad \text{ for all } \varphi \in Y\},
\een
where $e:[0,\tau]\times X\times Y \to \VR$ is a function that is linear with respect to the
last argument. The Lagrangian $(t,u,p) \mapsto  G(t,u,p): [0, \tau ] \times X \times Y \to \VR$ associated with \eqref{E:gt} is defined by
\ben
G(t,u,q) = f(t,u) + e(t,u,q).
\een
With this Lagrangian the set $E(t)$ can be expressed in terms of the Lagrangian $G$ as follows
\ben
E(t) = \{u^t\in X: \partial_p G(t,u^t,0)(\varphi)=0\; \text{ for all } \varphi\in Y\}
\een
and $g$ can be written as a minimax (see \cite{a_DEST_2017a})
\ben
g(t) = \inf_{\varphi\in X}\sup_{\psi\in Y} G(t,\varphi,\psi) = \inf_{\varphi\in E(t)} G(t,\varphi,0).
\een
We will provide new conditions (see Hypothesis (H3)) under which the function $g$ is right differentiable. The pertinence of the result is illustrated by applying it to a finite dimensional problem and
a shape optimisation  problem.

The problem of finding the right derivative of \eqref{E:gt} arises naturally when deriving optimality conditions appearing in equality constrained finite and infinite dimensional control and shape optimisation problems. Accordingly it has been studied by many authors before and sufficient conditions, even with inequality constraints are known; see, e.g., the review article \cite{a_BOSH_1998a}. Often for inequality constrained problems suitable \emph{constraint qualifications}  (e.g. Robinson's constraint qualification \cite{a_RO_1980a})  are required which impose a certain regularity on minimisers; see \cite{a_IZSO_2001a,a_ZOKU_1979a,a_KU_1976a}. In \cite{a_LEMA_1980a} the right differentiability is examined in infinite dimensions under the assumption that the elements of $E(t)$ arise from convex optimisation problems; see also \cite{a_SH_1995a} and \cite{a_BOCO_1996a,a_BOCO_1996b} for results in infinite dimension.

In case $E(t)$ is independent of $t$ let us mention the early work of J.~M.~Danskin \cite{b_DA_1967a,a_DA_1966a} where a maximum function with respect to  a  parameter was studied. When the solution of the maximum problem (and similarly minimum problem) is not unique, then a natural non-differentiability arises. In this case only directional derivatives or sub-differentials are computable. We also refer to the monographs \cite{b_DEMALO_2014a,b_GRTE_2001a,b_RI_1998a} and references therein. In the review article \cite{a_BOSH_1998a} and also the book chapters \cite[Chap. 4]{b_BOSH_2000a} and \cite[Chap. 2]{b_ITKU_2008a} several conditions for right differentiability of $g$ are given (see also references therein). In particular first and second order expansions of value functions are studied using second order conditions.

As mentioned before second order analysis can be used to obtain differentiability of the optimal
solution $u^t$ and hence differentiability of the value function $g$. Let us mention \cite{a_GUKU_2018a} where the differentiability of the value function with respect to Dirichlet data of a tracking-type cost function constrained by a semilinear parabolic PDE is studied. A key ingredient is a H\"older estimate of order  $1/2$  of the optimal control with respect to the Dirichlet data.

The differentiability of parametric minimax functions under saddle point assumptions has been studied in \cite{a_COSE_1985a} by Correa and Seeger and was subsequently extended and applied to shape optimisation problems by Delfour and Zol\'esio in \cite{a_DEZO_1989a}.
For nonlinear equality constraints this saddle point assumption is unfortunately often not satisfied.

In \cite{a_ST_2015a, phd_ST_2014a} an approach to the differentiability of a minimax without a saddle point assumption for the Lagrangian  was presented. An extension to the multivalued case can be found in \cite[Thm. 4.1]{a_DEST_2017a} and \cite[Thm. 2 and Thm. 3]{c_DEST_2016a}. In addition in \cite[Thm. 3.1]{a_DEST_2017a} and \cite[Thm. 1]{c_DEST_2016a} also the singleton case was revisited and extended by introducing an extra term. For applications to the single valued case of this approach we refer to \cite{a_ST_2015a} and also \cite{a_KAKUST_2018a} and \cite{a_ST_2016a}. In this context let us also mention the approaches of \cite[p.54, Thm. 4.6]{phd_ST_2014a} and  \cite[Thm. 3.3]{c_DE_2018b} (see also \cite{c_DE_2018a}), where a Lagrangian approach using an unperturbed adjoint variable is proposed for the single-valued case. The adjoint method \cite[Thm. 3.3]{c_DE_2018b} has also been recently used in \cite{a_GAST_2019a.13420v2} and \cite{a_GAST_2019b.10775v2} to compute topological derivatives and, in addition, also a thorough  comparison with the averaged adjoint method is provided. From this it appears, at least in the context of computing topological derivatives, that the averaged adjoint method seems favorable, since a larger class of cost functionals could be treated.  This may be rooted in the infinite dimensionality of the problem.

The key idea of the averaged adjoint approach is to replace the perturbed standard adjoint by the so-called \emph{averaged adjoint state equation}.  This allows to deal with
non-convex objective functions and non-linear state equations. Let us also mention the variational approach of
\cite{a_ITKUPE_2006a}  where another approach is proposed to show the differentiability of a minimax by using some sort of second order expansion. Both approaches have in common that they bypass the computation of the derivative of the control-to-state operator. Although both approaches are from its nature very different we will show in this paper how
they can effectively be combined to establish yet another even more powerful new theorem on the  differentiability of the minimax.

Our result gives new easy to check conditions and generalises results in
\cite{a_DEST_2017a}. The target applications of our theorem are the \textit{shape sensitivity analysis} yet the result can also be applied to optimal control problems in general.

\subsection*{Notation}

Let $f:[0,\tau]\times  \mathcal U_1\times \mathcal U_2  \rightarrow \VR$ be a function defined on the Cartesian product of the interval $[0,\tau]$, $\tau >0$ and the open subsets  $\mathcal U_1\subset X$  and  $\mathcal U_2\subset Y$ of normed spaces. Then we define
for $(t,u,p)\in [0,\tau)\times \mathcal U_1\times\mathcal U_2$ and $v\in  X$ and $w\in Y$ the following one sided directional derivatives
\ben
\begin{split}
    \partial_tf(t,u,p)&:= \lim_{h \searrow 0} \frac{f(t+h,u,p)-f(t,u,p)}{h},\\
    \partial_{u} f(t,u,p)(v)&:= \lim_{ h  \searrow 0} \frac{f(t, u +hv,p )-f(t,u,p)}{h},\\
    \partial_{p} f(t,u,p)(w)&:= \lim_{ h \searrow 0 } \frac{f(t,u,p+h w)-f(t,u,p)}{h},
\end{split}
\een
provided the limits on the right hand side exist, respectively.  The notation $h\searrow 0$ indicates that $h\rightarrow 0$ under the condition $h>0$.

We equip $\VR^d$ with the Euclidean norm $\|\cdot\|$ and denote by $\|A\|$
the corresponding operator norm for $A\in \VR^{d\times d}$.

Throughout the paper, we will use the terminology \emph{state equation} and \emph{adjoint state equation}.

\section{Minimax theorem via the averaged adjoint equation}

\subsection{Averaged adjoint equation}
Let $X,Y$ and $G$ be as in the introduction. We will henceforth assume
that $g(t)$ is finite for all $t\in [0,\tau]$.

\begin{definition}
    We introduce for $t \in [0,\tau]$ the set of minimisers
    \ben
    X(t) := \{u^t \in E(t): \; \inf_{u\in E(t)} f(t,u) = f(t,u^t) \}.
    \een
\end{definition}
Notice that $X(t)\subset E(t)$ and that $X(t)=E(t)$ whenever $E(t)$ is a singleton. However, in general $X(t)$ and $E(t)$ do not need to coincide.  The definition of the averaged adjoint equation requires that the set of states is not empty:
\begin{assumption*}[H0]
    For all $t\in [0,\tau]$ we have $X(t)\not= \emptyset$.
\end{assumption*}
Before we can introduce the averaged adjoint equation we need the following hypothesis.
\begin{assumption*}[H1]
    For all $t \in [0,\tau]$ and $(u^0, u^t) \in X(0)\times E(t)$ we assume:
    \begin{enumerate}
        \item [\textup{(i)}]
            For all $p \in Y$, the mapping $s\mapsto G(t , su^t + (1-s)u^0), p):[0,1]\to \VR$ is absolutely continuous.
        \item [\textup{(ii)}]
            For all $(\varphi,q)\in X\times Y$ and almost all $s\in (0,1)$
            the function
            \ben
            s \mapsto \partial_u G(t, su^t + (1-s)u^0,p)(\varphi):[0,1]\to \VR
            \een
            is well-defined and belongs to $L_1(0,1)$.
    \end{enumerate}
\end{assumption*}
\begin{remark}\label{rem:crucial}
    Notice that item (i) implies that for all $t \in [0,\tau]$,  $(u^0, u^t)\in X(0)\times E(t)$  and $p\in Y$,
    \ben
    \label{eq:mean_value_affine}
    G(t,u^t,p) = G(t,u^0,p) +\int_0^1  \partial_u G(t, su^t + (1-s)u^0,p)(u^t-u^0) \, ds.
    \een
    This follows at once by applying the fundamental theorem of calculus to $s\mapsto G(t, su^t + (1-s)u^0, p)$ on $[0,1]$.
\end{remark}

The following gives the definition of the adjoint and averaged adjoint equation; see \cite{a_ST_2015a}.
\begin{definition}[Averaged adjoint equation]
    Let $\tilde X\subset X$ be a linear subspace. Given $t\in [0,\tau]$ and $(u^0,u^t)\in E(0)\times E(t)$, the
    \emph{averaged adjoint state equation} is defined as follows: find $q^t \in Y$, such that
    \ben\label{eq:AAE}
    \int_0^1 \!\!\partial_u G(t,su^t + (1-s)u^0 ,q^t)(\varphi)\, ds=0 \quad \text{ for all } \varphi\in \tilde X.
    \een
    For every triplet $(t, u^0,u^t)$ the set of solutions to \eqref{eq:AAE} is denoted by $Y(t,u^0,u^t)$.
\end{definition}
\begin{definition}
    The standard adjoint $p^t\in X$ is defined by  $\partial_u G(t,u^t,p^t)(\varphi)=0$ for all
    $\varphi\in \tilde X$ and the set of adjoints associated with $(t,u^t)$ is denoted $Y(t,u^t)$.
\end{definition}
Notice that $ Y(0, u^0) = Y(0,u^0,u^0)$ for all $u^0\in E(0)$, that is, the averaged adjoint equation
reduces to the usual adjoint equation. The averaged adjoint equation allows us to express the Lagrangian at time $t$ solely through the Lagrangian evaluated at  $(t,u^0,q^t)$.

\begin{lemma}\label{lem:reparam_lagrange}
    Let  $t\in (0,\tau]$. Then for all $(u^0,u^t)\in E(0)\times E(t)$ with $u^t-u^0\in \tilde X$, and $q^t\in Y(t,u^0,u^t)$, we have
    \ben
    \label{eq.simple2}
    G(t,u^t,q^t)  = G(t,u^0,q^t).
    \een
\end{lemma}
\begin{proof}
    This follows directly from \cref{rem:crucial} noting that $u^t-u^0\in \tilde X$ is an admissible test function in \eqref{eq:AAE} and hence the last term in \eqref{eq:mean_value_affine} vanishes.
\end{proof}

\begin{remark}
    Notice that \eqref{eq.simple2} holds for all $t>0$, but not necessarily at $t=0$. The reason behind
    this is a discontinuity at $t=0$. Let $u^t\in E(t)$ and $\bar u^0\in E(0)$ with $u^0\ne \bar u^0$ and let $q^t\in  Y(t,\bar u^0, u^t)$. Set
    $f_1(t):= G(t,u^t,q^t)$ and $f_2(t):= G(t,\bar u^0,q^t)$ for $t\in [0,\tau]$. Then from \eqref{eq.simple2} we obtain
    \ben
    f_1(t) = f_2(t) \quad \text{ for all } t>0,
    \een
    but $f_1$ and $f_2$ do not coincide at $t=0$ unless $f_1(0) = f(0,u^0)=f(0,\bar u^0) = f_2(0)$. However, if we also let $\bar u^0_t\in E(0)$, such that $\bar u^0_0=u^0$, then the functions
    $f_1(t):=G(t,u^t,q^t)$ and $f_2(t):=G(t,\bar u^0_t,q^t)$ will coincide at $t=0$. This observation is important for our main theorem (\cref{thm.one}); see also Hypothesis~(H3).
\end{remark}

\begin{corollary}\label{cor:g_inf}
    For all $t\in (0,\tau]$, $(u^0,u^t)\in E(0)\times X(t)$ with $u^t-u^0\in \tilde X$, and for all $ q^t\in Y(t,u^0,u^t)$, we have
    \ben\label{eq:g_Gt}
    g(t) = G(t,u^t,q^t) = G(t,u^0,q^t).
    \een
\end{corollary}
\begin{proof}
    Let $t\in (0,\tau]$, $(u^0,u^t)\in E(0)\times X(t)$, and $ q^t\in Y(t,u^0,u^t)$ be given. Since $u^t\in X(t)$ we obtain by definition
    \ben\label{eq:cor1}
    g(t) = \inf_{u\in E(t)} G(t,u,0) = G(t,u^t,0).
    \een
    On the other hand since $X(t)\subset E(t)$, we have $u^t \in E(t)$ and thus
    $G(t,u^t,0) = G(t,u^t,q^t).$ Now we can apply \cref{lem:reparam_lagrange} to obtain
    \ben\label{eq:cor2}
    G(t,u^t,q^t) = G(t,u^0,q^t).
    \een
    Finally \eqref{eq:cor1} and \eqref{eq:cor2} together imply \eqref{eq:g_Gt}.
\end{proof}

\begin{remark}
    Notice that in our setting the test space of the adjoint and averaged adjoint equation might be smaller than the space of definition of the
    parametrised Lagrangian, that is, $\tilde X\ne X$ in general. This is for instance the case when solving the Dirichlet Laplacian where the test space would be $H^1_0$ and the trial space $H^1$. We refer to the last section for an example.
\end{remark}

\begin{remark}\label{rem:help}
    Let $G:[0,\tau]\times X \times Y \rightarrow \VR$ be a Lagrangian and $u\in X$. Assume  that $\partial_t G(0,u,p)$ exists for all $p\in  Y$. Then $p\mapsto \partial_t G(0,u,p):Y \rightarrow \VR$ is affine.
\end{remark}

\subsection{A new minimax theorem for Lagrangians}
The next theorem gives new sufficient conditions for $g$ to be right differentiable at $t=0$. Our theorem extends \cite{a_COSE_1985a}
and complements \cite[Theorem 4.1]{a_DEST_2017a} for functions $G$ that are Lagrangians. We will give a new theorem which provides new sufficient conditions for which the limit
\ben
dg(0):= \lim_{t\searrow 0}\frac{g(t)-g(0)}{t}
\een
exists, where $g$ is given by \eqref{E:gt}.

\begin{theorem}
    \label{thm.one}
    Let $G$ be a Lagrangian and suppose that Hypotheses \textup{(H0)--(H1)} and the following conditions are satisfied.
    \begin{enumerate}
        \item[\textup{(H2)}]
            For all  $u\in X(0) $  and all $p\in Y(0,u)$,
            $\partial_t G(0,u,p)$ exists;

        \item[\textup{(H3)}]    For every null-sequence $(t_n)$, $t_n\in (0,\tau]$, there exist $u^0\in X(0)$ and $p^0\in Y(0,u^0)$, a subsequence $(t_{n_k})$, elements   $(u^0_{t_{n_k}}, u^{t_{n_k}})\in E(0) \times X(t_{n_k})$,  $u^{t_{n_k}}-u^0_{t_{n_k}}\in \tilde X$ and $q^{t_{n_k}}\in Y(t_{n_k},u^0_{t_{n_k}},u^{t_{n_k}})$, such that
            \[
                \liminf_{
                    k\to \infty} \frac{G(t_{n_k},u^0_{t_{n_k}},q^{t_{n_k}})
                -G(0,u^0_{t_{n_k}},q^{t_{n_k}})}{t_{n_k}} \ge \partial_t G(0, u^0, p^0).
            \]

        \item [\textup{(H4)}]
            For every $ u^0\in X(0)$ there is  a  suboptimal path  $ t\mapsto \bar u^t:[0,\tau]\rightarrow X $  satisfying, $\bar u^0=u^0$,  $\bar u^t-u^0\in \tilde X$, $\bar u^t\in E(t)$ and
            \[ \lim_{t\searrow 0}\frac{\|\bar u^t -u^0\|_{X}}{t^{1/2}} = 0 \]
            and for all $p^0\in Y(0,u^0)$,
            \ben\label{eq:hoelder_ubar}
            \limsup_{t\searrow 0}\frac{G(t,\bar  u^t,p^0)-G(0,\bar u^t,p^0)}{t} \le  \partial_t G(0,u^0,p^0).
            \een

        \item [\textup{(H5)}] For every $u^0\in X(0)$ and every $p^0\in Y(0,u^0)$,
            \ben
            |G(0,u,p^0)-G(0,u^0,p^0) - \partial_{u} G(0,u^0,p^0)(u-u^0)| = \mathcal O (\|u-u^0\|^2_{X}).
            \een
    \end{enumerate}
    Then the one sided derivative $dg(0)$ exists and we find $u^0\in X(0)$ and $p^0\in Y(0,u^0)$, such that
    \ben\label{eq:thm1_dg}
    dg(0)  = \partial_t G(0,u^0,p^0)
    \een
    and we have the bound
    \ben\label{eq:thm1_dg_bound}
    \inf_{u\in X(0)}\inf_{p\in Y(0,u)} \partial_tG(0,u,p)       \le  dg(0)  \le \inf_{u\in X(0)}\sup_{p\in Y(0,u)} \partial_tG(0,u,p).
    \een
    If, in addition, for all $u\in X(0)$ the set $Y(0,u)=\{p^0(u)\}$ is a singleton, then
    \ben\label{eq:thm1_dg2}
    dg(0)  =  \inf_{u\in X(0)}  \partial_t  G(0, u, p^0(u)) = \partial_t G(0,u^0,p^0).
    \een

\end{theorem}
Before we turn our attention to the proof of this theorem let us make a few remarks.

\begin{remark}
    Let us give a guideline on how Hypothesis (H3) can be verified in practice. Let a null-sequence $(t_n)$ and  $(u^{t_n})\in X(t_n)$ be given. Typically one can use compactness arguments to find $u^0\in X(0)$ and a subsequence  (denoted the same)  such that $u^{t_n} \to u^0$ in some topology on $X$ (e.g. weak or strong). Then
    one constructs $u^{t_n}_0\in E(0)$, such that $u^{t_n}_0\to u^0$ and $Y(t_n,u^{t_n}_0,u^{t_n})\ne \emptyset$. Then it only remains to verify that there is a sequence $q^{t_n}\in Y(t_n,u^{t_n}_0,u^{t_n}) $ of averaged adjoints that converges to some element $q^0\in Y(0,u^0)$.

\end{remark}

\begin{remark}
    \begin{itemize}
        \item In contrast to previous theorems (see, e.g., \cite{a_DEST_2017a,c_DEST_2016a}) we allow the points $u^0_{t}\in E(0)$ to change when $t$ approaches zero.  The idea is to choose
            $u_0^t$ in such a way that the averaged adjoint variable
            $q^t\in Y(t,u_t^0,u^t)$ exists. We will illustrate this with a nonconvex example in \cref{sec:finite}.
        \item  Assumption (H3) extends Hypothesis (H3) of \cite[Thm. 4.1]{a_DEST_2017a} by perturbing the elements $u_0\in X(0)$. This allows us to treat examples where the original averaged adjoint variable is not well-defined; see  \cref{subsec:counter}.
        \item  Assumptions (H4) and (H5) in \cref{thm.one} follow ideas used in \cite{a_ITKUPE_2006a} (see also their follow up work \cite{a_KAKU_2014a,a_KAKU_2011a}) and replace Hypothesis~(H4) of \cite[Thm. 4.1]{a_DEST_2017a}. The main motivation and advantage  is that  we do not need to assume that we find for all $u^0\in X(0)$ a continuous path $[0,\tau]\rightarrow \VR^d:t\mapsto u^t$ with $u^t\in X(t)$, which might be difficult to check or might be even  false (see the example in \cref{subsec:counter}). We also refer to \cite{a_ITKUPE_2006a} for an example
            where $t\mapsto u^t$ is in fact not differentiable, but where (H4),(H5) are satisfied.  We also note that the other recent articles such as  \cite[Thm. 4.1.]{a_DEST_2017a}, \cite[Thm. 4.1.]{c_DE_2018b},\cite[Thm. 6.1 and Thm. 6.2]{c_DE_2018a} always work with optimal paths. Therefore the replacement of this condition with the present one is crucial.
        \item Let us mention other results related to ours. In \cite[Thm. 4.4]{a_BOSH_1998a} the differentiability of $t\mapsto g(t)$ is proved under the assumption that the minimisation problem \eqref{E:gt} is convex and that there is an $o(t)$ optimal path $u^t$, such that $\|u^t-u^0\|_X = o(t)$. This latter condition is similar to condition (H3), however, we only require the existence of $u^t\in E(t)$ with $1/2$ H\"older continuity. However, the result \cite[Thm. 4.4]{a_BOSH_1998a} also includes inequality constraints; see also \cite{a_SH_1995a} and \cite{a_BOCO_1996a,a_BOCO_1996b}.
    \end{itemize}
\end{remark}

We split the proof of this theorem in  two lemmas  in which we prove upper and lower bounds for
the following liminf and limsup of the differential quotients of $g$:
\begin{gather}
    \underline{dg}(0) := \liminf_{t\searrow 0} \frac{g(t)-g(0) }t
    \quad \text{and}\quad
    \overline{dg}(0) := \limsup_{t\searrow 0} \frac{g(t)-g(0) }t.
\end{gather}

\begin{lemma}\label{lem:lower_bound}
    Assume that $G$ satisfies Hypotheses \textup{(H0)--(H3)}. Then
    \ben\label{eq:concl_1}
    \exists u^0\in X(0) , \;\exists p^0\in Y(0,u^0), \quad \underline{dg}(0) \ge  \partial_t  G(0,u^0,p^0).
    \een
    In particular, we have
    \ben\label{eq:lower_bound}
    \underline{dg}(0) \ge \inf_{u\in X(0)}\inf_{q\in Y(0,u)} \partial_t G(0,u,q).
    \een
\end{lemma}
\begin{proof}
    Let $(t_n)$, $t_n>0$ be a null-sequence, such that
    \begin{gather*}
        \lim_{n\rightarrow \infty}\frac{g(t_n)-g(0) }{t_n}
        = \underline{dg}(0).
    \end{gather*}
    From \cref{cor:g_inf} we get for all  $t\in (0,\tau]$, $(u^0,u^t)\in E(0)\times X(t)$ with $u^t-u^0\in \tilde X$,   and for all $ q^t\in Y(t,u^0,u^t)$,
    \ben\label{eq:ineq_ge_one}
    g(t) = G(t,u^t,q^t) = G(t,u^0,q^t).
    \een
    In addition, from the definition of $X(0)$, we have for every $p\in  Y$,
    \ben\label{eq:ineq_ge_two}
    g(0) \le G(0,u^0,0) = G(0,u^0,p).
    \een
    Equations \eqref{eq:ineq_ge_one} and \eqref{eq:ineq_ge_two} together yield: for all  $t\in (0,\tau]$, $(u^0,u^t)\in E(0)\times X(t)$ and $q^t\in Y(t,u^0,u^t)$,
    \ben\label{eq:chain_le}
    \frac{g(t)-g(0) }t
    \ge \frac{G(t, u^0  ,q^t)
    -G(0, u^0 ,q^t)}t.
    \een
    From Assumption (H3): for every null-sequence $(t_n)$ there exist $u^0\in X(0)$ and $p^0\in Y(0,u^0)$, such that there is a subsequence $(t_n)$, indexed the same,  $(u^0_{t_n},u^{t_n})\in E(0)\times X(t_n) $ and $q^{t_n}\in Y(t_n,u^0_{t_n},u^{t_n})$, such that
    \[
        \lim_{n\rightarrow \infty}\frac{g(t_n)-g(0) }{t_n}  \stackrel{\eqref{eq:chain_le}}{\ge}\liminf_{n\rightarrow \infty}\frac{G(t_n, u^0_{t_n} ,q^{t_n})
        -G(0, u^0_{t_n} ,q^{t_n})}{t_n} \ge  \partial_t  G(0,u^0,p^0)
    \]
    which is precisely \eqref{eq:concl_1}.
\end{proof}

\begin{remark}
    The lower bound obtained in \eqref{eq:lower_bound} is weaker than the one of e.g. \cite[Prop. 2.3]{a_SH_1995a}. In fact, there it is proven that
    \ben
    \inf_{u\in X(0)}\sup_{p^0\in Y(0,u)}  \partial_t  G(0, u, p^0) \le \underline{dg}(0).
    \een
    However, in this proposition it is assumed that  the optimisation problem appearing in the definition of $g(t)$ is a convex optimisation problem,  which together with assumptions on $X(0)$ leads to a lower bound
    for $\underline{dg}(0)$. Nevertheless our bound together with the bound proved in the following lemma will still lead to the right differentiability of $g$. The price we have to pay here is that the final expression of the derivative is not a minmax anymore and thus contains less information.
\end{remark}

\begin{lemma}\label{lem:upper_bound}
    Assume that $G$ satisfies Hypotheses \textup{(H0)--(H2)} and \textup{(H4)--(H5)}. Then
    \ben\label{eq:concl_2}
    \forall u^0\in X(0), \; \forall p^0\in Y(0,u^0),\quad   \overbar{dg}(0) \le \partial_t G(0,u^0,p^0).
    \een
    In particular,
    \ben\label{eq:upper_bound}
    \overline{dg}(0) \le \inf_{u\in X(0)}\sup_{q\in Y(0,u)} \partial_t G(0,u,q).
    \een
\end{lemma}
\begin{proof}
    Let $(t_n)$, $t_n>0$ be a null-sequence, such that
    \begin{gather*}
        \lim_{n\rightarrow \infty}\frac{g(t_n)-g(0) }{t_n}
        = \overline{dg}(0).
    \end{gather*}
    We have for all $t\in [0,\tau]$, $u^t\in E(t)$ and all $p\in  Y$,
    \ben\label{eq:ineq_le_two}
    g(t) = \inf_{u\in E(t)} G(t, u,0) \le G(t,u^t,0)  = G(t,u^t,p).
    \een
    As a result for all  $t\in (0,\tau]$, $(u^0,u^t)\in X(0)\times E(t)$, and all $p\in  Y$,
    \ben\label{eq:le_gt_g0}
    \begin{aligned}[t]
        \frac{g(t)-g(0)}t & \stackrel{\eqref{eq:ineq_le_two}}{\le} \frac{G(t,u^t,p)
        -G(0,u^0,p)}t \\
        & = \frac{G(t,u^t,p)
            -G(0,u^t,p)}t + \frac{G(0,u^t,p)
        -G(0,u^0,p)}t.
    \end{aligned}
    \een
    To further estimate the right hand side note that it follows from  Assumption (H5):  For every $u^0\in X(0)$ and every $p^0\in Y(0,u^0)$,
    \ben\label{eq:est_G_second}
    |G(0,u,p^0)-G(0,u^0,p^0) - \partial_{u} G(0,u^0,p^0)(u-u^0)| = \mathcal O (\|u-u^0\|^2_{X}).
    \een
    On the other hand by definition of the adjoint state $p^0\in Y(0,u^0)$,
    \ben\label{eq:est_G_second2}
    \partial_{u} G(0,u^0,p^0)(\varphi)=0 \quad \text{ for all } \varphi\in \tilde X.
    \een
    Moreover thanks to  Assumption (H4)  we find  $t\mapsto \bar u^t:[0,\tau]\rightarrow X$,  such that $\bar u^0 = u^0$, $\bar u^t-u^0\in \tilde X$, $\bar u^t\in E(t)$ and $\|\bar u^t-u^0\|_{X} = o(t^{1/2})$. Hence combining \eqref{eq:est_G_second} and \eqref{eq:est_G_second2}  gives that for every $u^0\in X(0)$ and every $p^0\in Y(0,u^0)$ there is a constant $C$ (depending on $u^0$ and $p^0$) such that for all small $t$, we have
    \ben\label{eq:est_G_second3}
    \begin{aligned}[t]
        |G(0,\bar u^t,p^0)-G(0,u^0,p^0)| & \stackrel{\eqref{eq:est_G_second2}}{=} |G(0, \bar u^t,p^0)-G(0,u^0,p^0) -  \partial_{u} G(0,u^0,p)(\bar u^t-u^0) | \\
        &  \stackrel{\eqref{eq:est_G_second}}{\le} C\|\bar u^t-u^0\|_{X}^2.
    \end{aligned}
    \een
    As a result
    \ben\label{eq:est_G_second4}
    \begin{aligned}[t]
        \left|\limsup_{t\searrow 0} \frac{G(0,\bar u^t,p^0)-G(0,u^0,p^0)}{t}\right| & \le \limsup_{t\searrow 0}\frac{|G(0,\bar u^t,p^0)-G(0,u^0,p^0)|}{t} \\
        & \stackrel{\eqref{eq:est_G_second3}}{\le} C\limsup_{t\searrow 0} \frac{\|\bar u^t-u^0\|_{X}^2}{t} \stackrel{\eqref{eq:hoelder_ubar}}{=} 0.
    \end{aligned}
    \een
    Therefore from \eqref{eq:le_gt_g0} for every $u^0\in X(0)$, there is $t\mapsto \bar u^t$ as before such that, for all $p^0\in Y(0,u^0)$,
    \ben
    \begin{aligned}[t]
        \overbar{dg}(0) & \le     \limsup_{t\searrow 0}\frac{G(t,\bar u^t ,p^0)
            -G(0,  \bar u^t , p^0)}t +  \limsup_{t\searrow 0}\frac{ G(0, \bar u^t ,p^0)
        -G(0, u^0, p^0)}t \\
        & \le \partial_t G(0,u^0,p^0)
    \end{aligned}
    \een
    which is precisely \eqref{eq:concl_2}.
\end{proof}

\begin{remark}
    A sufficient condition in the convex case to derive   an upper bound  for $\overline{dg}(0)$ (even with inequality constraints) as in the previous lemma is the assumption that every point in $X(0)$ satisfies the \emph{Robinson constraint qualification} (see \cite[p. 5]{b_ITKU_2008a} and \cite{a_RO_1980a} for a definition). We refer to \cite[Prop. 2.1]{a_SH_1995a} which is due to \cite{a_LEMA_1980a}.
\end{remark}

\paragraph{Proof of \cref{thm.one}}
Let Hypotheses~\textup{(H0)--(H5)} hold true.    Combining \eqref{eq:concl_1} and \eqref{eq:concl_2} shows there exist $u^0\in X(0)$ and $p^0\in Y(0,u^0)$, such that
\ben\label{eq:final}
\partial_t G(0,u^0,p^0) \le \underline{dg}(0) \le \overbar{dg}(0) \le \partial_tG(0,u^0, p^0),
\een
which implies that $dg(0)$ exists and is equal to $\partial_t G(0,u^0,p^0)$.
If for all $u^0\in X(0)$ the set $Y(0,u^0)=\{p^0(u^0)\}$ is a singleton, then we obtain from \eqref{eq:concl_1} and \eqref{eq:final}, that for all $\tilde u\in X(0)$,
\ben
\inf_{u\in X(0)}\partial_t G(0,u,p^0(u)) \le  \partial_tG(0,u^0,p^0) \le dg(0) \le \partial_tG(0,\tilde u, p^0(\tilde u)).
\een
Taking the infimum over $\tilde u$ in $X(0)$ yields \eqref{eq:thm1_dg2}.

\paragraph{Alternative upper bound}
Let us finish this section with  an upper bound  for $dg(0)$, which can be derived by
replacing Hypotheses (H4),(H5) by the following relaxed Hypothesis~(H4'). Its advantage over
(H4),(H5) is that no H\"older continuity of $u^t$ is needed, but only convergence. However, the bound is weaker than the one of \cref{lem:lower_bound}. Compare also with the general result \cite[Thm. 4.5]{a_BOSH_1998a}.

\begin{assumption*}[H4']
    For every null-sequence $(t_n)$, $t_n>0$ and every $u^0 \in X(0)$, there exist  $p^0 \in Y(0,u^0)$,
    a subsequence $(t_{n_k})$ of $(t_n)$,
    $u^{t_{n_k}}\in E(t_{n_k})$,  and
    $q^{t_{n_k}}\in Y(t_{n_k},u^0,u^{t_{n_k}})$, such that
    \begin{gather*}
        \limsup_{
            k\to \infty} \frac{G(t,u^0,q^{t_{n_k}})
        -G(0,u^0,q^{t_{n_k}})}{t_{n_k}}
        \le \partial_t G(0, u^0, p^0).
    \end{gather*}
\end{assumption*}

\begin{lemma}\label{lem:bound_H4}
    Let Hypotheses~(H0)--(H3) and (H4') be satisfied. Then
    \ben
    \forall u^0\in X(0),\; \exists p^0\in Y(0,u^0), \quad \overline{dg}(0) \le \partial_t G(0,u^0,p^0).
    \een
    In particular,
    \ben\label{eq:improved_bound}
    \overline{dg}(0) \le \inf_{u\in X(0)} \sup_{p\in Y(0,u)}\partial_t G(0,u,p).
    \een
\end{lemma}
\begin{proof}
    Let $(t_n)$, $t_n>0$ be a null-sequence, such that
    \begin{gather*}
        \lim_{n\rightarrow \infty}\frac{g(t_n)-g(0) }{t_n}
        = \overline{d} g(0).
    \end{gather*}
    By definition we have for all $t>0$, $u^t\in E(t)$, $u^0\in X(0)$ and $q^t\in Y(t,u^0,u^t)$:
    \ben
    g(t) \le G(t,u^t,0) = G(t,u^t,q^t) = G(t,u^0,q^t),
    \een
    where in the last step we used \cref{cor:g_inf}. Hence we obtain from
    Hypothesis (H4') that we find $(t_n)$ and $u^0\in X(0)$, a subsequence,  denoted  the same, an element $p^0\in Y(0,u^0)$, elements $u^{t_n}\in E(t_n)$ and $q^{t_n}\in Y(t_n,u^0,u^{t_n})$, such that
    \ben\label{eq:bound_cor}
    \begin{split}
        \overline{dg}(0) = \lim_{n\to\infty}\frac{g(t_n)-g(0) }{t_n}   & \le \limsup_{n\to\infty}\frac{G(t_n, u^0 ,q^{t_n})-G(0,  u^0, q^{t_n})}{t_n}  \le \partial_t G(0,u^0,p^0).
    \end{split}
    \een
    Hence in particular for every $u^0\in X(0)$
    \ben
    \overline{dg}(0) \le \sup_{p\in Y(0,u^0)} \partial_tG(0,u^0,p).
    \een
    Taking the infimum over $u^0$ yields \eqref{eq:improved_bound}.
\end{proof}

As said before the statement of the previous lemma is weaker than the one of \cref{lem:upper_bound}. However,  if for all  $u^0\in X(0)$, the set $Y(0,u^0)$ is a singleton we obtain right differentiability of $g$.

\begin{corollary}
    Let Hypotheses~(H0)--(H3) and (H4') be satisfied. Assume that for all $u\in X(0)$ the set
    $Y(0,u)=\{p^0(u)\}$ is a singleton. Then $g$ is right differentiable and
    \ben
    dg(0)  =  \inf_{u\in X(0)}  \partial_t  G(0, u, p^0(u)).
    \een

\end{corollary}
\begin{proof}
    This directly follows from the proof of \cref{lem:bound_H4} (equation \eqref{eq:bound_cor}) and \cref{lem:lower_bound}.
\end{proof}

\section{Application to a finite dimensional problem}\label{sec:finite}
In this section we study a simple finite dimensional minimisation problem for which
we can apply \cref{thm.one}. The following example is a generalisation of  the one considered \cite[Sec. 6.3]{c_DE_2018a};  see also \cite[p.143]{b_DE_2012a}. We also refer to \cite{a_KU_2007a}, where existence  of  an optimisation problem with a quadratic cost function and quadratic separable inequality constraints is studied.
\subsection{Problem formulation}
Given two symmetric matrices $A,Q\in \VR^{d\times d}$ we define
\ben\label{problem:finite}
f(u) :=  Qu\cdot u, \qquad E:=\{u\in \VR^d:\; Au\cdot u =1\}
\een
and consider the minimisation problem
\ben\label{min:J}
\inf_{u\in E} f(u).
\een
The following assumption guarantees that \eqref{min:J} admits at least one solution.
\begin{assumption}\label{ass:finite}
    We assume that the pair of symmetric matrices $(Q,A)$ satisfies one of the following two conditions:
    \begin{itemize}
        \item[(a)] $Q$ is positive definite and there is $u\in \VR^d$ with $Au\cdot u>0$.
        \item[(b)] $Q$ is arbitrary and $A$ positive definite.
    \end{itemize}
\end{assumption}

\begin{remark}
    We note that the case of nonsymmetric $A$ and $Q$ can be reduced to the symmetric case.
    Indeed suppose $(Q,A)$ are nonsymmetric. Then $\hat A:= \frac12(A+A^\top)$ and $\hat Q:=\frac12(Q+Q^\top)$ satisfy (a) (resp. (b)) if and only if $(Q,A)$ satisfy (a) (resp. (b)). Hence we could
    work with $(\hat Q,\hat A)$ instead.
\end{remark}

\begin{lemma}
    Let $(Q,A)$ satisfy \cref{ass:finite}(a) or (b). Then the minimisation problem \eqref{min:J} admits
    a solution.
\end{lemma}
\begin{proof}
    If $A$ is positive definite, then it is readily checked that $E$ is compact.  If $A$ is indefinite, then $E$ need not to be bounded, but in this case $Q$ is positive definite.
Hence in either cases \eqref{min:J} is finite and a minimiser exists. \end{proof}

\begin{remark}
    We note that if \cref{ass:finite}(a) or (b) is satisfied, then the
    set $E$ is a $d-1$ dimensional embedded $C^\infty$-submanifold of $\VR^d$. This is a consequence
    of the regular value theorem; see \cite[Thm.9.3, p.243]{b_AMES_II_2006a}. In fact, setting
    $f(u):= Au\cdot u-1$, then we have $\nabla f(u) =  2Au\ne 0$ for all $u\in E$, so that $0$ is a regular
    value of the function $f$.
\end{remark}

Let us continue with a few examples of matrices that satisfy (a) or (b).

\begin{example}\label{ex:1}
    \begin{itemize}
        \item[(i)] A pair $(Q,A)$ satisfying (a) is
            \ben
            Q=\begin{pmatrix}
                1 & 0\\
                0 & 1
                \end{pmatrix}, \qquad A = \begin{pmatrix}
                1 & 0\\
                0 & -1
            \end{pmatrix}.
            \een
            In this case $E=\{(x,y)^\top: x^2-y^2=1
            \}$ is a hyperbola.
        \item[(ii)] Another example of $(Q,A)$ satisfying (a) is
            \ben
            Q=\begin{pmatrix}
                1 & 0\\
                0 & 1
                \end{pmatrix}, \qquad A = \begin{pmatrix}
                0 & 0\\
                0 & 1
            \end{pmatrix}.
            \een
            In this case $E=\{(x,y)^\top: y^2=1\}=\VR\times \{1,-1\}$ are two lines parallel to the
            $x$-axis.
        \item[(iii)]
            A pair $(Q,A)$ for which (b) is satisfied is given by
            \ben
            Q=\begin{pmatrix}
                1 & 0\\
                0 & 0
                \end{pmatrix}, \qquad A = \begin{pmatrix}
                2 & 0\\
                0 & 1
            \end{pmatrix}.
            \een
            In this case $E=\{(x,y)^\top: 2 x^2+y^2=1\}$ is an ellipse.
    \end{itemize}
\end{example}
If (a) or (b) are satisfied, then \eqref{min:J} admits at least one minimiser
$u\in E$. The Lagrange multiplier rule shows that we find $p \in \VR$, such that
\ben\label{eq:Q_A}
Qu + p Au=0.
\een
So $\lambda := -p$ is a generalised eigenvalue for the matrices $(Q, A)$. It also follows from \eqref{eq:Q_A} and $Au\cdot u=1$ that
\ben
\lambda = -p = Qu\cdot u.
\een
From this it follows that if (a) holds, then $0<Qu\cdot u=-p$, so $p\ne 0$.  If (b) holds, then
$p=0$ is possible.

\subsection{Perturbation and Lagrangian}
We now consider the following perturbation of \eqref{problem:finite}
\ben\label{problem:finite_pert}
f(t,u) := Q(t)u\cdot u, \qquad E(t) := \{u^t\in \VR^d:\; A(t)u^t\cdot u^t =1\},
\een
where $Q,A:[0,\tau] \to \VR^{d\times d}$  are  matrix functions satisfying the following assumption.
\begin{assumption}\label{ass:ii}
    Let $Q,A:[0,\tau] \to \VR^{d\times d}$ be continuously differentiable functions, such that $Q(t)$ and $A(t)$ are symmetric for all $t\in [0,\tau]$ and the pair $(Q(0),A(0))$ satisfies \cref{ass:finite}, (a) or (b).
\end{assumption}
\begin{example}
    A pair $(Q,A)$ satisfying the previous assumption is given by
    \ben
    Q(t) = \begin{pmatrix}
        1 & 0 \\
        0 & 1
        \end{pmatrix}, \qquad A(t) = \begin{pmatrix}
        t & 0 \\
        0 & 1
    \end{pmatrix}.
    \een
    In this case $E(t) = \{(x,y)^\top:\; tx^2+y^2 =1\}$ is an ellipse for $t>0$, but
    $E(0)=\{(x,\pm 1)^\top: x\in \VR\}$ consists of two lines parallel to the $x$-axis.
    We illustrate the set $E(t)$ for various $t>0$ in \cref{eq:ellipse_lines}.
\end{example}

\begin{figure}
    \includegraphics[width=\textwidth]{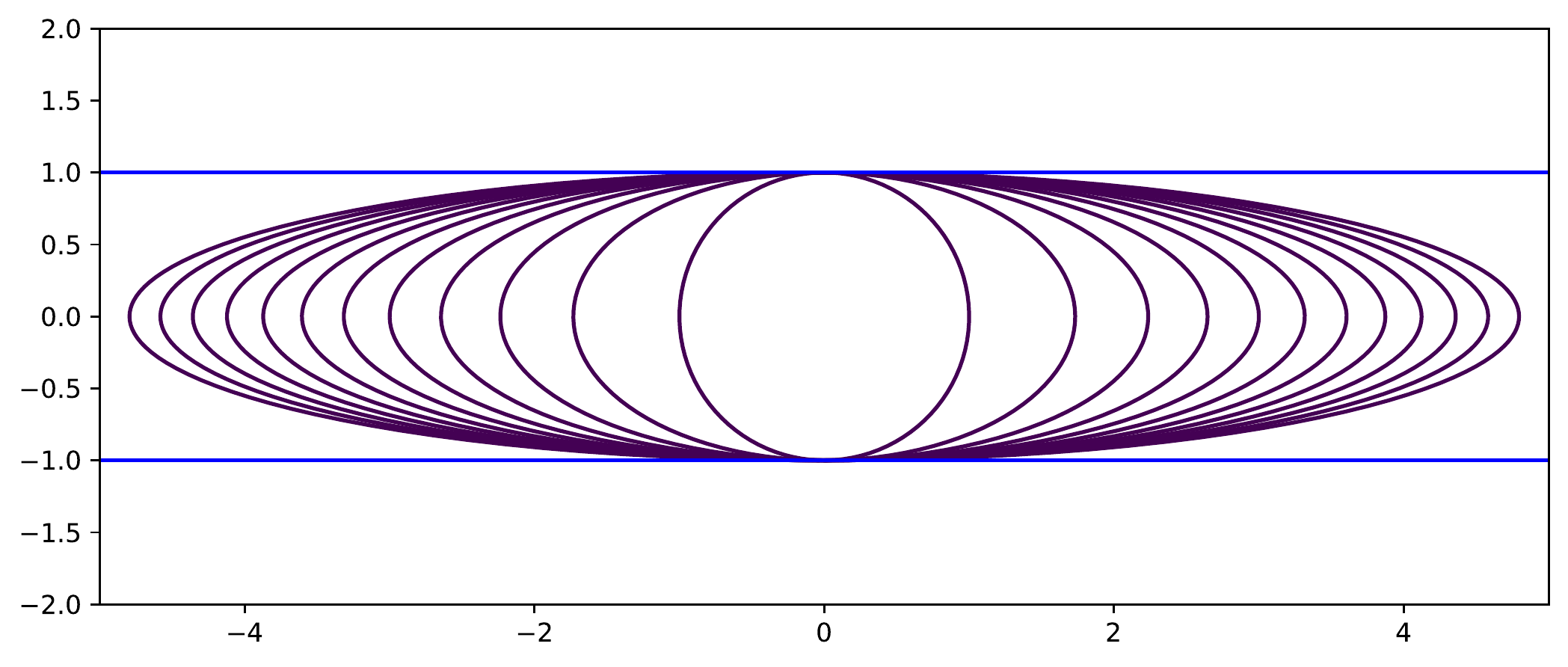}
    \caption{Ellipse converging to the two blue lines}
    \label{eq:ellipse_lines}
\end{figure}

We will show that if \cref{ass:ii} holds true, then
\ben\label{eq:min_finite_dim}
g(t) := \inf_{u\in E(t)} Q(t)u \cdot u
\een
is right differentiable at $t=0^+$ by applying \cref{thm.one}.

The Lagrangian $G:[0,\tau]\times \VR^d \times \VR \to \VR$ associated with the minimisation problem \eqref{eq:min_finite_dim} reads
\ben\label{E:lagrangian_finite}
G(t,u,p) := Q(t)u\cdot u + p(A(t)u\cdot u-1).
\een
By the Lagrangian multiplier rule we find for every minimiser $u^t\in X(t)$ a number $p^t\in \VR$, such that
\ben\label{eq:!_A_t}
Q(t)u^t + p^t A(t)u^t=0.
\een
It also follows from \eqref{eq:!_A_t} and $A(t)u^t\cdot u^t=1$, that
\ben
g(t) = Q(t)u^t\cdot u^t = -p^t.
\een
This shows that the set $Y(t,u^t)=\{p^t\}$ is a singleton and also
\ben
X(t) = \{u^t\in \VR^d:\; Q(t)u^t + p^t(t) A(t)u^t=0, \quad A(t)u^t\cdot u^t=1\}.
\een
The averaged adjoint equation associated with two states  $(u^0, u^t)\in E(0)\times E(t)$  reads: find $q^t \in \VR$ such that
\ben\label{eq:averaged_example_finite}
Q(t)(u^t+u^0) + q^tA(t)(u^t+u^0) = 0.
\een
The existence of $q^t$ can not be guaranteed for all pairs $(u^0,u^t)$; see \cref{subsec:counter} for a counter example. However, \cref{thm.one} only requires the existence of the averaged adjoint variable for certain pairs of states.

\begin{remark}
    We note that if \cref{ass:finite}, (b) is satisfied we have
    \ben
    g(t) = \inf_{u\in E(t)} f(t, u) = \inf_{\substack{u\in \VR^d\\ u\ne 0}} \frac{Q(t)u\cdot u}{A(t)u\cdot u}.
    \een
    Therefore in this case we can also apply Danskin's theorem (see, e.g., \cite[p. 524, Thm. 2.1]{b_DEZO_2011a}, \cite[p.140, Thm.5.1]{b_DE_2012a} or \cite[Thm. 4.1]{a_BOSH_1998a}) to prove the right differentiability of $g$ at $t=0$ (in fact, one could even consider a more general $f(t,u)$ in this case). \cref{ass:finite}, (a) does not allow this simplification since $A$ is not necessarily positive definite in this case. Indeed consider the pair $(Q,A)$ from \cref{ex:1}, (i) (which satisfies (a)):
    \ben
    \inf_{\substack{u\in \VR^2\\ u\ne 0}}\frac{Qu\cdot u}{Au\cdot u} = \inf_{\substack{u\in \VR^2\\ u\ne 0}} \frac{x^2+y^2}{x^2-y^2} = -\infty,
    \een
    while $\argmin_{u\in E} x^2+y^2=\{(-1,0)^\top,(1,0)^\top\}$ and hence $\min_{u\in E} Qu\cdot u=\min_{u\in E} x^2+y^2=1$  is finite.

\end{remark}

\subsection{Verification of the Hypotheses}
We now verify Hypotheses~(H0)--(H5) for the Lagrangian $G$ in \eqref{E:lagrangian_finite} with $\tilde X=X=\VR^d$ and $Y=\VR$  and a real number $\tau>0$ which has to be chosen sufficiently small. In view of \cref{ass:finite} it is clear that Hypothesis~(H0) is satisfied. Hypotheses~(H1) and (H2) are also obvious since $A$ and $Q$ are differentiable.

\paragraph{Verification of Hypothesis~(H3)}
To this end, we first prove the following lemma.
\begin{lemma}\label{lem:X_bounded}
    \begin{itemize}
        \item[(i)] For sufficiently small $\tau$ the set $S_\tau := \cup_{t\in [0,\tau]} X(t)$ is bounded.
        \item[(ii)] For every null-sequence $(t_n)$ and $u^{t_n}\in X(t_n)$,  there is a subsequence still indexed the same, and $u^0\in X(0)$, such that $u^{t_n} \to u^0$ as $n\to \infty$.
    \end{itemize}
\end{lemma}
\begin{proof}
    We  first show the boundedness of $S_\tau$ for $\tau$ small if either (a) or (b) of \cref{ass:finite} hold.
    First suppose that \cref{ass:finite}, (a) is satisfied.  Notice that $Q(t)$ is uniformly positive definite for all small $t$. Then by definition of $u^t\in X(t)$, we have for all $t\in [0,\tau]$,
    \ben\label{eq:bound_Q}
    \alpha\|u^t\|^2\le Q(t)u^t\cdot u^t \le Q(t)u\cdot u \quad \text{ for all } u\in \VR^d, \; A(t)u\cdot u=1
    \een
    for some $\alpha >0$.  Now pick any $u^0\in \VR^d$ with $A(0)u^0\cdot u^0=1$. By continuity we find $\tau >0$ and $c>0$ such that $A(t)u^0\cdot u^0\ge c>0$ for all $t\in [0,\tau]$ and hence
    $\hat u^t:= u^0/\sqrt{A(t)u^0\cdot u^0}$ satisfies $A(t)\hat u^t\cdot \hat u^t=1$ and thus $\hat u^t\in E(t)$. Then plugging $\hat u^t$ into \eqref{eq:bound_Q} we obtain
    \ben
    \alpha \|u^t\|^2 \le \frac{Q(t)u^0\cdot u^0}{A(t)u^0\cdot u^0} \le c^{-1}\max_{t\in [0,\tau]}\|Q(t)\|\|u^0\|^2.
    \een
    Thus $S_\tau$ is bounded.

    Now suppose that \cref{ass:finite}, (b) holds. Since $A(0)$ is positive definite and since $A(\cdot)$ is continuous also $A(t)$ is positive definite provided $t$ is small enough. So we find $\alpha>0$, such that $\alpha\|u\|^2\le A(t)u\cdot u$ for all $u\in \VR^d$ and all small $t$. Therefore for all $u^t\in E(t)$ we have
    $\alpha\|u^t\|^2 \le A(t)u^t\cdot u^t=1$ which implies that $S_\tau$ is bounded.

    The proof of $(ii)$ follows by standard arguments and  hence is omitted.
\end{proof}

\begin{lemma}\label{eq:averaged_adjoint_finite}
    For every null-sequence $(t_n)$, we find $u^0\in X(0)$ and $p^0\in Y(0,u_0)$, and a
    subsequence (denoted the same), elements $(u_{t_n}^0,u^{t_n})\in E(0)\times X(t_n)$ and
    $q^{t_n}\in Y(t_n,u_{t_n}^0,u^{t_n}) $, such that
    \begin{align}
        q^{t_n} \to p^0 \quad &\text{ as } n\to \infty,\\
        u^0_{t_n} \to u^0 \quad &\text{ as } n\to \infty.
    \end{align}
\end{lemma}
\begin{proof}
    Let $(t_n)$ be an arbitrary null-sequence and let $(u^{t_n})$,  $u^{t_n}\in X(t_n)$ be given. Thanks to the previous \cref{lem:X_bounded} we find $u^0\in X(0)$ and a subsequence of $(t_n)$, still indexed the same, such that $u^{t_n} \to u^0$ as $n\to \infty$. By the Lagrange multiplier rule we find $p^0\in Y(0,u^0)$, such that
    \ben
    Q(0)u^0 + p^0A(0)u^0 =0, \quad A(0)u^0\cdot u^0=1.
    \een
    Since $A(0)u^0\cdot u^0=1$, we have $A(0)u^{t_n}\cdot u^{t_n}>0$ for $n$ large enough; therefore
    \ben
    u^0_{t_n} := \frac{1}{\sqrt{A(0)u^{t_n}\cdot u^{t_n}}} u^{t_n} \in E(0),
    \een
    is well-defined.
    It is clear that $u^0_{t_n}\to u^0$ as $n\to \infty$. By construction $u^0_{t_n}$ and $u^{t_n}$ are linearly dependent. Since $u^{t_n}\in X(t_n)$ the Lagrange multiplier rule shows $Q(t_n)u^{t_n} + q^{t_n}A(t_n)u^{t_n}=0$ for some $q^{t_n}\in \VR$, and thus also
    \ben\label{eq:averaged_finite_proof}
    Q(t_n)(u^0_{t_n} + u^{t_n}) + q^{t_n}A(t_n)(u^0_{t_n} + u^{t_n})=0,
    \een
    which is the averaged adjoint equation \eqref{eq:averaged_example_finite} associated with the pair $(u^0_{t_n},u^{t_n})$.
    Since $A(t_n)(u^0_{t_n} + u^{t_n})\ne 0$ for $n$ large, it follows from \eqref{eq:averaged_finite_proof} that
    \ben
    q^{t_n} =  -\frac{Q(t_n)(u^{t_n} + u^0_{t_n})\cdot (u^{t_n} + u^0_{t_n})}{A(t_n)(u^{t_n} + u^0_{t_n})\cdot (u^{t_n} + u^0_{t_n})} \to -\frac{Q(0)u^0\cdot u^0}{A(0)u^0\cdot u^0} = p^0.
    \een
    This shows  Hypothesis~(H3)  and finishes the proof.
\end{proof}
\paragraph{Verification of Hypothesis~(H4)}

The verification of Hypothesis~(H4) is a simple application of the inverse function theorem.

\begin{lemma}\label{lem:suboptimal}
    For every $u^0\in E(0)$ we find a differentiable function $t\mapsto u^t:[0,\tau]\to \VR^d$ with $u^t\in E(t)$ for all $t\in (0,\tau]$.
\end{lemma}
\begin{proof}
    We argue similarly as in \cref{eq:averaged_adjoint_finite}. Let $u^0\in \VR^d$ with $A(0)u^0\cdot u^0 =1$ be given. Again by continuity we find $\tau>0$ and $c>0$, such that $A(t)u^0\cdot u^0\ge c$ for all $t\in [0,\tau]$. Now
    \ben
    u^t:= \frac{1}{\sqrt{A(t)u^0\cdot u^0}}
    \een
    belongs to $E(t)$ and is right differentiable at $t=0$.
\end{proof}

\paragraph{Verification of Hypothesis~(H5)}
To check Hypothesis~(H5) we compute for all $u\in \VR^d$ and $q\in \VR$,
\ben\label{eq:partial_u_G}
\partial_u G(0,u,p) = (Q(0) + pA(0))u.
\een
Hence by the fundamental theorem of calculus and \eqref{eq:partial_u_G} we obtain for all $u^0\in X(0)$, $u\in \VR^d$ and $p^0\in Y(0,u^0)$,
\ben
\begin{split}
    |G(0,u,p^0)-G(0,u^0,p^0) & - \partial_u G(0,u^0,p^0)(u-u^0)|\\
    & = |\int_0^1 \partial_u G(0,su+(1-s)u^0,p^0)(u-u^0)-  \partial_u G(0,u^0,p^0)(u-u^0)\;ds|\\
    & \le |u-u^0|^2 2\|Q(0) + p^0A(0)\|
\end{split}
\een
and this verifies Hypothesis~(H5).

\paragraph{Application of \cref{thm.one}}

Now we have verified all the Hypotheses of \cref{thm.one} and we obtain the following theorem.
\begin{theorem}\label{thm:finite_dimensional}
    Let $Q,A:[0,\tau] \to \VR^{d\times d}$ be two continuously differentiable
    functions satisfying \cref{ass:ii}. Then the function $g$ is right differentiable at $t=0$ and we find $u^0\in X(0)$, such that
    \ben
    dg(0) = \inf_{u\in X(0)} (Q'(0)+p^0A'(0))u\cdot u = (Q'(0)+p^0A'(0))u^0\cdot u^0,
    \een
    where $(u^0,p^0)\in \VR^d\times \VR$ satisfies:
    \begin{align*}
        Q(0)u^0+p^0A(0)u^0 &=0, \qquad p^0 = -Q(0)u^0\cdot u^0, \\
        A(0)u^0\cdot u^0 & = 1.
    \end{align*}
\end{theorem}

We can even obtain more using the differentiability of the suboptimal paths $u^t$ of \cref{lem:suboptimal} and the arguments of \cref{lem:lower_bound}. In contrast to the previous theorem the following lemma holds for all elements $u^0\in X(0)$.

\begin{lemma}
    Let $u^0\in X(0)$. Let $v\in \VR^d$ be any vector, such that
    there is $(u^t)$ with $u^t\in E(t)$ satisfying
    \ben
    v = \lim_{t\searrow 0}\frac{u^t-u^0}{t} \quad \text{ and } \quad A'(0)u^0\cdot u^0\ne 0.
    \een
    Then we have
    \ben
    dg(0) \le \partial_t G(0,u^0,\mu) = (Q'(0)+\mu A'(0))u^0\cdot u^0,
    \een
    where $\mu$ is given by
    \ben
    \mu := -\frac{Q(0)u^0\cdot v}{A(0)u^0\cdot v}.
    \een
\end{lemma}
\begin{proof}
    Let $u^0\in X(0)$ and let $u^t$ be a differentiable path with $v:= \lim_{t\searrow 0}(u^t-u^0)/t$. By definition $A(t)u^t\cdot u^t=1$ and hence $A'(0)u^0\cdot u^0 + 2 A(0)u^0\cdot v=0$. In view of $A'(0)u^0\cdot u^0 \ne0$, we conclude
    \ben\label{eq:not_eq_zero}
    A(0)u^0\cdot v \ne 0.
    \een
    Now we choose $q^t\in \VR$ such that
    \ben\label{eq:averaged_enhanced}
    \begin{split}
        \int_0^1 & \partial_u G(t,su^t + (1-s)u^0,q^t)(u^t-u^0)  \\
        & =  2(Q(t) + q^t A(t))(u^t-u^0)=0.
    \end{split}
    \een
    It is readily checked using \eqref{eq:not_eq_zero} that $A(t)(u^t+u^0)\cdot (\frac{u^t-u^0}{t})\ne0$ for all $t$ small and thus we obtain from \eqref{eq:averaged_enhanced},
    \ben
    \begin{split}
        q^t & =-\frac{Q(t)(u^t+u^0)\cdot (\frac{u^t-u^0}{t})}{A(t)(u^t+u^0)\cdot (\frac{u^t-u^0}{t})}
        \to -\frac{Q(0)u^0\cdot v}{A(0)u^0\cdot v} =:\mu.
    \end{split}
    \een
    Now we apply the mean value theorem to obtain
    \ben
    G(t,u^t,q^t)-G(t,u^0,q^t) = \int_0^1 \partial_u G(t,su^t + (1-s)u^0,q^t)(u^t-u^0)\;ds =0,
    \een
    where the last equality follows from the definition of $q^t$. Hence we obtain that for every null-sequence $t\searrow 0$ and every $u^0\in X(0)$, we find $u^t\in E(t)$ and $q^t\in \VR$, such that
    \ben
    G(t,u^t,q^t)=G(t,u^0,q^t).
    \een
    Hence we can use the same arguments as in \cref{lem:lower_bound} to conclude $dg(0) \ge \partial_t G(0,u^0,\mu)$.

\end{proof}

\subsection{On the condition (H3) and non-existence of averaged adjoints}\label{subsec:counter}
Let us give an explicit example where  for all sufficiently small $t>0$  and any pair $(u^0,u^t)\in X(0)\times X(t)$ the set
of averaged adjoints $Y(t,u^0,u^t)=\emptyset$ is empty. However, for every $u^t\in X(t)$, we find $u^0\in X(0)$ and $u^t_0\in E(0)$, such that $u^t_0\to u^0$ and $Y(t, u^t_0, u^t)\ne \emptyset$. Therefore Hypothesis (H3) is necessary in some cases and can not be simplified.

Consider again $g$ defined by \eqref{eq:min_finite_dim} with
\begin{align}
    E(t) & = \{u=(x,y)^\top\in \VR^2:\; A(t)u\cdot u =1\}, \\
    A(t)& := \begin{pmatrix}
        1 & t \\
        t & 1
        \end{pmatrix}, \qquad Q = \begin{pmatrix}
        1 & 0 \\
        0 & 2
    \end{pmatrix}.
\end{align}
Clearly \cref{ass:ii} is satisfied for this example and hence $g$ is right differentiable thanks to \cref{thm:finite_dimensional}. At $t=0$ we have
\begin{align}
    E(0)  = \{(x,y)^\top \in \VR^2:\; x^2+y^2=1\}, \quad X(0)  = \{(\pm 1,0)^\top\}, \quad Y(0,(\pm 1,0)^\top)=\{1\}
\end{align}
and for $t>0$ and  $u^t\in X(t)$ we find $p^t\in Y(t,u^t)$ solving
\ben
Qu^t + p^t A(t)u^t=0 \quad \Leftrightarrow \quad A(t)^{-1}Qu^t + p^t u^t=0.
\een
We compute
\ben
A(t)^{-1}Q = \frac{1}{1-t^2} \underbrace{\begin{pmatrix} 1 & - 2t \\ -t & 2 \end{pmatrix}}_{=:\tilde A(t)}.
\een
The eigenvalues of $\tilde A(t)$ are given by
\ben
\lambda^{\pm}(t) = \frac12\left(3\pm \sqrt{1+8t^2}\right)
\een
and therefore $\{\frac{1}{1-t^2}\lambda^-(t),\frac{1}{1-t^2}\lambda^+(t)\}$ are the eigenvalues of $ A(t)^{-1}Q$ and the eigenspaces are one dimensional. Moreover, $g(t) = \frac{1}{1-t^2}\lambda^-(t)$. For small $t>0$ we have $\lambda^+(t) >2$ and $\lambda^-(t) < 1$. The corresponding eigenvalue equations lead to
\begin{align}
    x(\lambda^\pm - 1) &= -2ty, \\
    y(\lambda^\pm - 2) & =  - tx.
\end{align}
So for $\lambda^+$ we obtain as eigenvector for $t>0$,
\ben\label{eq:at}
u^t_+ = \begin{pmatrix}
    a^t \\ 1
\end{pmatrix}, \quad a^t = \frac{-2t}{\lambda^+(t)-1}<0,
\een
and $a^t \to 0$ for $t\searrow 0$. Similarly, for $\lambda^-$ we obtain
\ben\label{eq:bt}
u_-^t = \begin{pmatrix}
    1 \\ b^t
\end{pmatrix}, \quad b^t = \frac{-t}{\lambda^-(t)-2}>0,
\een
and  $b^t \to 0$ for $t\searrow 0$. It follows that
\ben
X(t) = \{\hat u_-^t,-\hat u_-^t\}, \quad    \hat u_-^t:= \frac{1}{\sqrt{A(t)u_-^t\cdot u_-^t}}u_-^t.
\een
Now let us check that the averaged adjoint equation is not solvable.
\begin{lemma}
    For $t>0$ small and for $(\hat u^t_-, \pm u^0)$ with $u^0=(1,0)^\top$, there is no $q^t\in \VR$, such that
    \ben\label{eq:averaged}
    Q(\hat u^t_- \pm u^0) + q^t A(t)(\hat u^t_- \pm u^0) =0.
    \een
    In particular $Y(t,u^0,u^t_-)=Y(t,-u^0,u^t_-)=\emptyset$ and no averaged adjoint  state  for such pairs  exists.
\end{lemma}
\begin{proof}
    Suppose $q^t\in Y(t,u^0,u_-^t)$ exists. Then in view of the definition \eqref{eq:averaged}, this means that $q^t\in \{-\frac{1}{1-t^2}\lambda^-(t),-\frac{1}{1-t^2}\lambda^+(t)\}$ and in view of \eqref{eq:averaged}
    \ben
    \begin{split}
        \hat u^t_-+u^0 = & \frac{1}{\sqrt{A(t)u_-^t\cdot u_-^t}}\begin{pmatrix} 1 \\ b^t \end{pmatrix} +  \begin{pmatrix} 1 \\ 0 \end{pmatrix}
    \end{split}
    \een
    must be an eigenvector of $A(t)^{-1}Q$. Since the eigenspaces are one dimensional, we must have one of the two cases:
    \begin{subequations}
        \ben\label{eq:case2}
        \hat u^t_- + u^0 = \alpha u^t_-,
        \een
        \ben\label{eq:case4}
        \hat u^t_-+u^0 = \alpha u^t_+,
        \een
    \end{subequations}
    for some $\alpha\in \VR$. By comparing the last component of the vectors of \eqref{eq:case2}, we see that equality in \eqref{eq:case2} can only happen for $\alpha = \frac{1}{\sqrt{A(t)u_-^t\cdot u_-^t}}$, which gives
    \ben
    \hat u^t_-+u^0 =  \hat u_-^t \quad \Rightarrow \quad u^0 =0
    \een
    and thus a contradiction.
    Similarly for \eqref{eq:case4}, we compare the last component and see that equality can only be true for
    $\alpha = \frac{b^t}{\sqrt{A(t)u^t_-\cdot u^t_-}}$, which leads to
    \ben
    \hat u^t_-+u^0 =  \frac{b^t}{\sqrt{A(t)u_-^t\cdot u_-^t}} \begin{pmatrix} a^t \\ 1\end{pmatrix}
    \een
    which is also impossible since $a^t\to 0$ and $b^t\to 0$ and hence the right vector
    goes to zero as $t\searrow 0$, however, the left hand side goes to $2u^0\ne0$. Therefore \eqref{eq:averaged} is not solvable and $Y(t,u^0,u^t_-)=\emptyset$. The same arguments show that $Y(t,-u^0,u^t_-)=\emptyset$.
\end{proof}

Despite this negative result, we can define for every $u^t\in X(t)$ the element \linebreak $u^0_t:= u^t/\sqrt{A(0)u^t\cdot u^t}\in E(0)$, which is linearly dependent on $u^t$ and hence $Y(t,u^0_t,u^t)\ne \emptyset$. Moreover, if $u^t$ converges, also $u^0_t$ converges.

\subsection{Second  order sufficient conditions}
Let us finish this section by making some remarks on second order analysis results. We notice that in \cite[Sec. 4.9.1., p. 365]{b_BOSH_2000a} equality constraints are treated using second order analysis. However, for our problem this is not applicable. In fact in \cite[Thm. 4.125, p. 365]{b_BOSH_2000a} the following problem is studied
\ben
g(t) = \inf_{u \in E(t)}f(t,u ), \quad E(t) = \{ u^t\in X:\; e(t,u^t)=0\},
\een
where $f:[0,\tau]\times X\to \VR$ and $e:[0,\tau]\times X\to Y$ are two times differentiable functions. Introduce the associated Lagrangian
\ben
G(t,u,p):= f(t,u) + \langle p, e(t,u)\rangle_{Y^*,Y}, \quad t\in [0,\tau], \; u\in X,\; p \in Y^*.
\een
Suppose that $\partial_u e(0,u^0):X\to Y$ is surjective, such that the Lagrange multiplier
$q^0\in Y(0,u^0)$ for every $u^0\in X(0)$ is unique. Then in \cite[Thm. 4.125, p. 365]{b_BOSH_2000a} it is proved that for given $u^0\in X(0)$ and $q^0\in Y(0,u^0)$, there exist locally unique solutions $(u^t,p^t)$ of
\begin{align}\label{eq:opt_example}
    \partial_u G(t,u^t,p^t)=   \partial_u f(t,u^t) + \langle p^t, \partial_u e(t,u^t)\rangle_{Y^*,Y} & = 0,\\
    e(t,u^t) & = 0,
\end{align}
provided the second order condition holds:
\ben
\exists \alpha >0, \; \partial_u^2 G(0,u^0,p^0)(v)(v)                                                        \ge \alpha \|v\|^2_X \quad \text{ for all } v\in \text{kern}(\partial_u e(0,u^0)).
\een
This is a consequence of the implicit function theorem. However, this does not work in our setting in general since the second order condition for \eqref{problem:finite_pert}  reads with $X=\VR^d$ equipped with the  Euclidean  norm $\|\cdot\|$, $Y=\VR$, $e(t,u) =A(t)u\cdot u-1$ and $f(t,u):= Q(t)u\cdot u$:
\ben\label{eq:second_order}
\exists \alpha >0, \; (Q(0) + p^0 A(0))v\cdot v  \ge \alpha \|v\|^2 \quad \text{ for all } v\in (A(0)u^0)^\bot.
\een
Then \eqref{eq:opt_example} would read:
\ben
(Q(t) + p^t A(t))u^t =0, \quad A(t)u^t\cdot u^t = 1.
\een
However, it is readily seen that \eqref{eq:second_order} cannot hold when the
eigenvalue $-p^0$ of $(Q(0),A(0))$ is not simple. Take for instance $A(0)=I$, $Q(0)=Q(0)^\top$, and assume $-p^0$ is not geometrically simple (i.e., the eigenspace has dimension $\ge 2$) eigenvalue of $Q(0)$.  Then \eqref{eq:second_order} cannot hold, since $\text{kern}(A(0)u^0) = (u^0)^\bot=\{v\in \VR^d:v\cdot u^0=0\}$ contains another eigenvector $v\in(u^0)^\bot$ associated with $-p^0$ and thus $Q(0)v\cdot v + p^0\|v\|^2 =0$.

\section{Application to a shape optimisation problem}
In this section we present another example where \cref{thm.one} is applicable.
In contrast to the previous section this example is infinite dimensional.
\subsection{Shape optimisation problem}
For every bounded Lipschitz domain $\Omega\subset \VR^d$, $d\in \{2,3\}$, we consider
\ben\label{eq:cost_J}
J(\Omega) := \inf_{u\in E(\Omega)} f(\Omega, u),
\een
where
\ben
f(\Omega,u) := \int_{\Omega} |u-u_r|^2 \; dx + \gamma \int_\Omega|\nabla u|^2\;dx,\quad \gamma >0, \;  u_r\in H^1(\VR^d),
\een
and $E(\Omega)$ comprises the set of solutions $u=u_\Omega\in H^1(\Omega)$ to the semilinear problem:
\ben\label{eq:state_weak}
\int_{\Omega} \nabla u\cdot \nabla \varphi  + \varrho(u) \varphi\; dx = \int_{\Omega} f\varphi\;dx
\een
for all $\varphi \in H^1_0(\Omega)$. We make the following assumption.
\begin{assumption}
    We assume that $\varrho:\VR\to \VR$ is a two times differentiable, Lipschitz continuous and strongly monotone function satisfying $\varrho(0)=0$.
    We also assume $f\in H^1(\VR^d)$.
\end{assumption}
Equation \eqref{eq:state_weak} cannot be uniquely solvable since no Dirichlet boundary conditions are prescribed. Given a function $u_r$ the cost $J(\Omega)$ measures the best solution
to \eqref{eq:state_weak} which is closest to $u_r$.  The set $E(\Omega)$ contains infinitely many elements and is nonconvex (unless $\varrho$ is linear).

\begin{lemma}\label{lem:existence}
    There exists a minimiser to the problem \eqref{eq:cost_J}.
\end{lemma}
\begin{proof}
    It is clear that $J(\Omega)$ is finite. Let $(u_n)$ be a minimising sequence in $H^1(\Omega)$, so that
    \ben\label{eq:min_sequence}
    J(\Omega) = \lim_{n\to \infty} \left(\int_{\Omega} |u_n-u_r|^2\;dx + \gamma \int_\Omega|\nabla u_n|^2\;dx\right).
    \een
    From this it immediately follows that $(u_n)$ is bounded in $H^1(\Omega)$. Hence due to Rellich's compactness theorem we find $u\in H^1(\Omega)$ and a subsequence, which is denoted the same, such that  $\nabla u_n \rightharpoonup \nabla u$ weakly in $L_2(\Omega)^d$ and $u_n \to u$ strongly in $L_2(\Omega)$. Hence we can pass to the limit $n\to \infty$ in
    \ben
    \int_\Omega \nabla u_n \cdot \nabla \varphi + \varrho(u_n)\varphi\;dx = \int_\Omega f\varphi\;dx \quad \text{ for all }\varphi \in H^1_0(\Omega)
    \een
    to conclude $u\in E(\Omega)$. In addition we infer from \eqref{eq:min_sequence}
    \ben
    \int_{\Omega}|u-u_r|^2\;dx + \gamma \int_\Omega|\nabla u|^2\;dx \le J(\Omega).
    \een
    This shows that $u$ is a minimiser and finishes the proof.
\end{proof}

Our goal is now to use \cref{thm.one} to show that the directional shape derivative of $J$
exists.
\begin{definition}
    The directional shape derivative of $J$ at $\Omega$ in direction $X\in C^{0,1}(\VR^d)^d$ is defined by
    \ben
    dJ(\Omega)(X) :=\lim_{t\searrow
    0}\frac{J((\Id+t X)(\Omega)) - J(\Omega)}{t}.
    \een
\end{definition}
Notice that the mapping $T_t := \text{id} + t  X:\VR^d\to \VR^d$ is a bi-Lipschitz mapping for all $|t|< 1/L(X)$, where $L( X)$ denotes the Lipschitz constant of $ X$.

Let us introduce the Lagrangian $G:H^1(\Omega)\times H^1_0(\Omega) \to \VR$ by
\ben
G(\varphi,\psi) := \int_{\Omega} |\varphi-u_r|^2 \; dx+\gamma \int_\Omega|\nabla\varphi|^2\;dx + \int_{\Omega}\nabla  \varphi \cdot \nabla \psi+\varrho(\varphi)\psi \; dx - \int_{\Omega}f\psi \; dx.
\een
\cref{lem:existence} guarantees that the set
\ben
X(\Omega) := \{ u\in E(\Omega): \inf_{\varphi \in E(\Omega)} G(\varphi,0)   =  G(u,0) \}
\een
is not empty. In the next paragraph we consider the perturbed versions of $E(\Omega)$ and $X(\Omega)$.
\subsection{Analysis of the perturbed problems}
We will show by applying \cref{thm.one}  that for  a  bounded Lipschitz domain $\Omega\subset \VR^d$ the directional shape derivative of $J$ exists. At first we consider any solution $u_t\in H^1(\Omega_t)$ defined on the perturbed domain $\Omega_t := T_t(\Omega)$ of
\ben\label{eq:state_ut}
\int_{\Omega_t} \nabla u_t \cdot \nabla \varphi+\varrho(u)\varphi\;dx = \int_{\Omega_t} f\varphi \;dx \quad \text{ for all } \varphi \in H^1_0(\Omega_t).
\een
Now since $\varphi \in H^1(T_t(\Omega))$ (resp. $\varphi \in H^1_0(T_t(\Omega))$) if and only if $\varphi\circ T_t \in H^1(\Omega)$ (resp. $\varphi\circ T_t \in H^1_0(\Omega)$) (see \cite[Thm. 2.2.2, p.52]{b_ZI_1989a}) changing variables in \eqref{eq:state_ut} shows that $u^t := u_t \circ T_t\in H^1(\Omega)$  solves
\ben\label{eq:state_ut_}
\int_{\Omega}A(t)\nabla  u^t \cdot \nabla \varphi + |\det(\partial T_t)| \varrho(u^t)\varphi \; dx = \int_{\Omega}f^t\varphi \; dx \quad \text{ for all  }\varphi \in H^1_0(\Omega),
\een
where
\ben
A(t) := \det(\partial T_t)\partial T_t^{-1}\partial T_t^{-\top}, \quad f^t := \det(\partial T_t) f\circ T_t.
\een
Therefore $u_t \in E(T_t(\Omega))$ if and only if $u^t:= u_t \circ T_t$ is in
\ben
E(t) := \left\{  u^t\in H^1(\Omega):\; u^t \text{ solves } \eqref{eq:state_ut_} \right\}.
\een
As a result we get for $t$ small
\ben
g(t) :=J(T_t(\Omega)) = \inf_{u\in E(t)} f(t,u),
\een
where
\ben
f(t,u):= \int_{\Omega} \det(\partial T_t)(u- u_r^t)^2 \; dx+\gamma \int_{\Omega} A(t)\nabla u^t\cdot \nabla u^t\; dx,
\een
where  $u_r^t:= u_r\circ T_t$.
This problem falls into the framework of \cref{thm.one}.

We recall the following proposition; see, e.g., \cite{phd_ST_2014a}.

\begin{proposition}\label{prop:phit}
    Let $\Dsf\subset \VR^d$ be a bounded and open set. Let $X:\VR^d\to \VR^d$ be a Lipschitz vector field. Then for $T_t:= id+tX$ there holds:
    \begin{itemize}
        \item[(i)] We have
            \begin{align*}
                \frac{ \partial T_t - I}{t} \rightarrow & \partial X && \text{ strongly in } C(\overline \Dsf, \VR^{d,d})\\
                \frac{ \partial T_t^{-1} - I}{t} \rightarrow & - \partial X && \text{ strongly in } C(\overline \Dsf, \VR^{d,d})\\
                \frac{ \det(\partial T_t) - 1}{t} \rightarrow & \Div(X) && \text{ strongly in } C(\overline \Dsf).
            \end{align*}
        \item[(ii)]     For all open sets $\Omega \subset \Dsf$ and all $\varphi\in H^1(\VR^d)$, we have
            \begin{align}
                \frac{\varphi\circ T_t  - \varphi}{t} \rightarrow & \nabla \varphi \cdot X
                && \text{ strongly in }  L_2(\Omega).
            \end{align}
    \end{itemize}
\end{proposition}

\begin{lemma}\label{lem:ut_to_u0}
    Let $u^0\in E(0)$ be given. Then we find a path $t\mapsto u^t: [0,\tau]\rightarrow H^1(\Omega)$ with $u^t\in E(t)$, $u^t-u^0\in H^1_0(\Omega)$, and a constant $c$, such that
    \ben
    \|u^t-u^0\|_{H^1} \le ct \quad \text{ for all } t\in [0,\tau].
    \een
\end{lemma}
\begin{proof}
    Let $u^0\in E(0)$ be given. By definition $u^0\in H^1(\Omega)$ solves:
    \ben
    \int_{\Omega} \nabla u^0 \cdot \nabla \varphi + \varrho(u^0)\varphi\; dx = \int_{\Omega} f \varphi \; dx \quad \text{ for all } \varphi \in H^1_0(\Omega).
    \een
    Set $g_0 := u^0|_{\partial \Omega}$ and consider for every $t\in [0,\tau]$: find $u^t \in H^1(\Omega)$, such that $u^t = g_0$ on $\partial \Omega$ and
    \ben\label{eq:solution_ut}
    \int_{\Omega} A(t)\nabla u^t \cdot \nabla \varphi + \Det(\partial T_t)\varrho(u^t)\varphi\; dx = \int_{\Omega} f^t \varphi \; dx \quad \text{ for all } \varphi \in H^1_0(\Omega).
    \een
    By construction $u^t$ is uniquely determined,  $u^t\in E(t)$, $u^t-u^0=0$ on $\partial \Omega$ for all $t$. We obtain from \eqref{eq:solution_ut}:
    \ben\label{eq:difference_ut}
    \begin{aligned}[t]
        \int_{\Omega} A(t)& \nabla (u^t-u^0)  \cdot \nabla \varphi\;dx +\int_\Omega \Det(\partial T_t)(\varrho(u^t)-\varrho(u^0))\varphi\;dx \\
        = & -\int_\Omega \Det(\partial T_t)\varrho(u^0)\varphi - f^t \varphi  + A(t)\nabla u^0\cdot \nabla \varphi\;dx\\
        = & -\int_\Omega (\Det(\partial T_t)-1)\varrho(u^0)\varphi - (f^t-f) \varphi  + (A(t)-I)\nabla u^0\cdot \nabla \varphi\;dx
    \end{aligned}
    \een
    for all $\varphi\in H^1_0(\Omega)$. Since $\varphi:= u^t-u^0$ is zero on $\partial \Omega$ we may use it as test function in \eqref{eq:difference_ut}. Hence using H\"older's inequality, the uniform monotonicity of $A$ and the monotonicity $(\varrho(x)-\varrho(y))(x-y)\ge 0$ for all $x,y\in \VR$ gives
    \ben
    \begin{aligned}[t]
        \int_{\Omega}& (u^t-u^0)^2 + |\nabla (u^t-u^0)|^2\;dx  \\
        & \le c(\|f^t-f\|_{C(\overline{\Omega})} + \|A(t)-I\|_{C(\overline{\Omega})^{d\times d}}\|u^0\|_{H^1(\Omega)} + \|\Det(\partial T_t)-1\|_{C(\overline{\Omega})}\|\varrho(u_0)\|_{C(\overbar{\Omega})}  )
    \end{aligned}
    \een
    and therefore it follows from \cref{prop:phit},
    \ben
    \|u^t-u^0\|_{H^1(\Omega)} \le c t \quad \text{ for all } t\in [0,\tau].
    \een
\end{proof}

\paragraph{Parametrised Lagrangian and  averaged adjoint}

We set $\tilde X =Y := H^1_0(\Omega)$ and $X:=H^1(\Omega)$. The parametrised Lagrangian  $G:[0,\tau]\times X\times Y\rightarrow \VR$ is given by
\ben\label{eq:parametrised_lagrangian}
\begin{aligned}[t]
    G(t,u,p) =&  \int_{\Omega} \det(\partial T_t)(u-u_r^t)^2 \; dx + \gamma\int_{\Omega}A(t)\nabla u\cdot \nabla u\;dx \\
    & + \int_{\Omega} A(t)\nabla u\cdot \nabla p + \Det(\partial T_t)\varrho(u)p - f^tp\; dx,
\end{aligned}
\een
where  we recall $A(t)= \Det(\partial T_t) \partial X^{-1}\partial X^{-\top}$, $u_r^t= u_r\circ T_t$ and $f^t = \Det(\partial T_t) f\circ T_t$.  It is noteworthy that in this example we have $\tilde X\ne X$. Using \cref{prop:phit} we see that $A(t)$ and $f^t$ are differentiable and we readily check for all $u,p\in H^1(\Omega)$:
\ben\label{eq:derivative_lagrangian}
\begin{aligned}[t]
    \partial_tG(0,u,p)  = & \int_\Omega \Div(X)(u-u_r)^2 - (u-u_r)\nabla u_r \cdot X + \gamma A'(0)\nabla u\cdot \nabla u\;dx \\
    & +\int_\Omega A'(0)\nabla u\cdot \nabla p + \Div(X)\varrho(u)p - f'p\;dx,
\end{aligned}
\een
where $A'(0) = \Div(X)I - \partial X -\partial X^\top$ and $f' = \Div(X)f + \nabla f\cdot X$.
It is also readily checked that assumptions (H0)--(H3) are satisfied. Moreover, since $\varrho$ is Lipschitz continuous we also readily check Hypothesis (H5).

The averaged adjoint associated with $u^t\in E(t)$ and $u^0\in E(0)$ reads: find $q^t \in H^1_0(\Omega)$, such that,
\ben\label{eq:averaged_example}
\begin{aligned}[t]
    \int_{\Omega} A(t)\nabla q^t & \cdot \nabla \varphi + \int_0^1 \Det(\partial T_t)\varrho'(su^t+(1-s)u^0)q^t \;ds \varphi \; dx \\
    & = -\int_{\Omega} \Det(\partial T_t)(2u_r^t- (u^t+u^0))\varphi\;dx-\int_{\Omega} \Det(\partial T_t) A(t)\nabla(u^t+u^0)\cdot \nabla \varphi \; dx
\end{aligned}
\een
for all $\varphi \in H^1_0(\Omega)$. It follows from the theorem of Lax-Milgram and the uniform coercivity of $A$ and $\Det(\partial T_t)$ that \eqref{eq:averaged_example} admits a unique solution.

\begin{lemma}\label{lem:compactnes_X}
    For every null-sequence $(t_n)$ and $(u^{t_n})$,  $u^{t_n}\in X(t_n)$ there is a subsequence $(t_{n_k})$ and $u^0\in X(0)$, such that
    \ben
    u^{t_{n_k}} \rightharpoonup u^0 \quad \text{ in } H^1(\Omega) \quad \text{ as } k\rightarrow \infty.
    \een
\end{lemma}
\begin{proof}
    Let $(t_n)$ be a null-sequence and $u^{t_n}\in X(t_n)$. By definition we have
    for all $n\ge 1$
    \ben\label{eq:le_J}
    J(T_{t_n}(\Omega)) \le \int_\Omega \det(\partial T_{t_n})(u-u_r^{t_n})^2\;dx + \gamma \int_\Omega  A(t_n)\nabla u\cdot \nabla u\;dx \quad \text{ for all } u\in E(t_n).
    \een
    Now fix $u^0\in E(0)$ and let $(\bar u^t)\in E(t)$ be as in \cref{lem:ut_to_u0}, such that
    $\bar u^t \to u^0$ in $H^1(\Omega)$. Plugging $\bar u^{t_n}$ into \eqref{eq:le_J} and using \cref{prop:phit}, we find $C>0$, such that
    \ben
    J(T_{t_n}(\Omega)) \le C \quad \text{ for all } n\ge 1.
    \een
    It follows in particular, since
    \ben
    J(T_{t_n}(\Omega)) =  \int_\Omega \det(\partial T_{t_n})(u^{t_n}-u_r^{t_n})^2\;dx + \gamma \int_\Omega   A(t_n)\nabla u^{t_n}\cdot \nabla u^{t_n}\;dx,
    \een
    that $(u^{t_n})$ is bounded. Hence there is a subsequence (denoted the same) and $u^0\in H^1(\Omega)$, such that $u^{t_n}\rightharpoonup u^0$ weakly in $H^1(\Omega)$ and $u^{t_n} \to u^0$ strongly in $L_2(\Omega)$. It is readily checked that by passing to the limit $n\to\infty $ that $u^0\in X(0)$, which finishes the proof.
\end{proof}

\begin{corollary}
    For every null-sequence $(t_n)$ and $u^{t_n}\in X(t_n)$, we find $u^0\in X(0)$ and $p^0\in Y(0,u^0)$, and a subsequence (denoted the same), elements  $(u^{t_n}_0,u^{t_n})\in E(0)\times X(t_n)$ and $q^{t_n}\in Y(t,u_0^{t_n},u^{t_n})$, such that $u^{t_n}-u^{0,t_n}\in H^1_0(\Omega)$ and
    \begin{align}
        u^0_{t_n} \rightharpoonup u^0 & \quad \text{ in } H^1(\Omega) \quad \text{ as } n \rightarrow \infty, \\
        q^0_{t_n} \rightharpoonup q^0 & \quad \text{ in } H^1_0(\Omega) \quad \text{ as } n \rightarrow \infty.
    \end{align}
\end{corollary}
\begin{proof}
    Thanks to \cref{lem:compactnes_X} we find for every null-sequence $(t_n)$ and $u^{t_n}\in X(t_n)$ a subsequence (denoted the same) and $u^0\in X(0)$, such that $u^{t_n} \rightharpoonup u^0$ weakly in $H^1(\Omega)$. Set $g_n := u^{t_n}|_{\partial \Omega}$ and consider: find $ u^0_{t_n} \in H^1(\Omega)$ with $u^0_{t_n} = g_n$ on $\partial \Omega$, such that
    \ben
    \int_{\Omega} \nabla u^0_{t_n} \cdot \nabla \varphi + \varrho(u^0_{t_n})\varphi \; dx = \int_{\Omega} f \varphi \; dx \quad \text{ for all } \varphi \in H^1_0(\Omega).
    \een
    By construction $u^0_{t_n}$ is uniquely determined and $u^0_{t_n}\in E(0)$. If we introduce $\tilde u_n := u^{t_n}-u^0_{t_n}$, then $\tilde u_n\in H^1_0(\Omega)$ and
    \ben\label{eq:I_An}
    \begin{aligned}[t]
        \int_{\Omega} & \nabla \tilde u_n \cdot \nabla \varphi + (\varrho(u^{t_n})-\varrho(u^0_{t_n}))\varphi\; dx  \\
        & =  \int_{\Omega}(I-A(t_n))\nabla u^{t_n}\cdot \nabla\varphi + (1-\Det(\partial T_t))\varrho(u^{t_n})\varphi  + (f^t-f)\varphi\;dx
    \end{aligned}
    \een
    for all $\varphi \in H^1_0(\Omega)$. So testing \eqref{eq:I_An} with $\varphi = \tilde u_n$ and using H\"older's inequality yields
    \ben\label{eq:I_An2}
    \|\tilde u_n\|_{H^1}\le c(\|A(t_n)-I\|_{C(\overbar \Omega, \VR^{d,d})} + \|1-\Det(\partial T_t)\|_{C(\overline{\Omega})}\|u^{t_n}\|_{H^1(\Omega)} + \|f^t-f\|_{L_2})\|u^{t_n}\|_{H^1}
    \een
    and since $(u^{t_n})$ is bounded in $H^1(\Omega)$ the result follows from \cref{prop:phit}. It follows that $u^{t_n}-u^{0,t_n} \rightharpoonup 0$ in $H^1(\Omega)$ and since $u^{t_n}\rightharpoonup u^0$ in $H^1(\Omega)$ we also conclude that $u^{0,t_n}$ converges weakly to $u^0$ in $H^1(\Omega)$. From this and \eqref{eq:averaged_example} it  is  also readily seen that the averaged adjoint $q^{t_n}$ for
    $(u^{t_n}_0,u^{t_n})$ exists and that $q^{t_n} \rightharpoonup q^0$ weakly in $H^1(\Omega)$ as $n\to \infty$.
\end{proof}

\paragraph{Application of \cref{thm.one}}
We have verified  that  assumptions (H0)--(H5) of \cref{thm.one} are satisfied for $G$ defined in \eqref{eq:parametrised_lagrangian} with $Y=\tilde X=H^1_0(\Omega)$ and $X=H^1(\Omega)$. Therefore we obtain the following theorem.
\begin{theorem}
    The right derivative of $g$ at $t=0^+$ exists and
    \ben
    dg(0)= \inf_{u\in X(0)}\partial_t G(0,u,p^0(u)),
    \een
    where $p^0(u)\in H^1_0(\Omega)$ solves
    \ben
    \int_{\Omega} \nabla p^0(u)\cdot \nabla \varphi +  \varrho'(u) p^0(u)\varphi \; dx = - \int_{\Omega} 2(u-u_r)\varphi  + 2\gamma \nabla u \cdot \nabla \varphi\;dx
    \een
    for all $\varphi \in H^1_0(\Omega)$ and $\partial_t G(0,u,p^0(u))$ is given by \eqref{eq:derivative_lagrangian}.
\end{theorem}

\section*{Conclusion}
In this paper we discussed a new minimax theorem and presented two examples. In both examples we could establish right differentiability of the corresponding value function.

In a future work it would be interesting to apply our result to optimal control problems with non-unique solution. Also to find an example where for the state $u^0\in X(0)$ the adjoint $Y(0,u^0)$
is not a singleton is still an open question.

\nocite{*}
\bibliographystyle{jnsao}
\bibliography{jnsao-2020-6034}

\end{document}